\documentclass{amsart}
\usepackage{amsmath, amssymb, amscd, amsthm}
\usepackage{stmaryrd}
\usepackage{latexsym}
\usepackage[matrix,arrow,knot,poly]{xy} 
\usepackage{pb-diagram}
\usepackage{eurosym}
\usepackage{color}
\usepackage{multicol}

\xyoption{all}

\newtheorem{theo}{Theorem}[section]

\newtheorem{lemma}[theo]{Lemma}

\newtheorem*{lem*}{Lemma}

\newcommand{\cc}{{\mathbb{C}}}                                     
\newcommand{\nn}{{\mathbb{N}}}                                     
\newcommand{\rr}{{\mathbb{R}}}                                     
\newcommand{\In} {{\subseteq}}                                     

\newcommand{\comment}[1]{}                                         

\newenvironment{aufg}[1]{\noindent \textbf{Aufgabe #1. }}

\begin{document}

\title[On the Kuratowski embedding]{On the finite-dimensional approximation of the Kuratowski embedding for compact Riemannian manifolds}

\author[Malte R\"oer] {Malte R\"oer \\
       Karlsruher Institut f\"ur Technologie \\
        Institut f\"ur Algebra und Geometrie \\
       Kaiserstrasse 89-93 \\
       76133 Karlsruhe \\
       Germany 
        }

\date{}

\begin{abstract}
In the proof of his systolic inequality, Gromov uses an isometric embedding of a Riemannian manifold $M$ into the Banach space of bounded functions on $M$, the so-called Kuratowski-embedding. Subsequently, it was shown by different authors that the Kuratowski-embedding can be approximated by bi-Lipschitz embeddings into finite-dimensional Banach spaces. We give a detailed proof for the existence of such finite-dimensional approximations along the lines suggested in \cite{MR2299729} and go on to discuss quantitative aspects of the problem, establishing for the dimension of the Banach space a bound which depends on curvature properties of the manifold.

\end{abstract}

\maketitle

\section{Introduction}
In \cite{MR697984}, Gromov proves his systolic inequality, showing that the least length of a non-contractible loop in an essential Riemannian manifold $M$ is bounded from above in terms of $\operatorname{Vol}(M)$ and a universal constant. In his proof, Gromov makes use of the isometric embedding
\begin{equation*}
M \rightarrow L^{\infty}(M) \, , \, m \mapsto \bigl( x \mapsto d(x,m) \bigr),
\end{equation*}
where $d$ is the metric induced on $M$ by the Riemannian metric. This is the so-called \emph{Kuratowski-embedding}. In \cite{MR2299729}, Guth states a lemma saying that the Kuratowski-embedding of a compact Riemannian manifold can be approximated by bi-Lipschitz-embeddings into finite-dimensional Banach spaces. Guth sketches a proof, which uses Toponogov's Theorem to show that for a compact Riemannian manifold $M$ and $\epsilon > 0$ there is a finite subset $S \subset M$ with the following property. For $x,y \in M$ there is an $s \in S$ with 
\begin{equation}
(1-\epsilon) \cdot d(x,y) \leq | d(x,s) - d(y,s)| \leq d(x,y).
\end{equation}
The first complete proof of this result appears in  \cite{MR2753701}. The proof given there, which involves the injectivity radius, the first variation formula and the mean value theorem, does not proceed along the lines suggested by Guth.

Another proof can be found in \cite{MR2803853}.

The purpose of this note is to give a detailed version of the proof for the approximation lemma sketched by Guth, and to pursue the following question: What can be said about the size of $S$, i.e., about the dimension of the Banach space $l^{\infty}(S)$ into which $M$ embeds with Lipschitz constants $1-\epsilon$ and $1$?
Given $\epsilon >0$ and a finite subset $S$ of a compact metric space $(X,d)$, we say that $S$ is \emph{$\epsilon$-good} if for any pair $(x,y) \in X \times X$ there is an $s \in S$ so that the estimate in equation (1) is valid, i.e., if for all $x,y \in X$ there is an $s \in S$ so that 
\begin{equation*}
(1- \epsilon) \cdot d(x,y) \leq | d(x,s) - d(y,s) |. 
\end{equation*}
Now define
\begin{equation*}
\kappa(X,d,\epsilon) := \operatorname{inf} \Bigl \{ |S| : S \subset X \mbox{ is $\epsilon$-good } \Bigr \}.
\end{equation*}
When $X$ is a manifold $M$ and $d$ is induced by a Riemannian metric $g$ on $M$, we also write $\kappa(M,g,\epsilon)$ for $\kappa(M, d, \epsilon)$.  
The main result of this paper is the establishment of an upper bound for $\kappa(M,g, \epsilon)$ where $(M,g)$ is a compact Riemannian manifold. In fact, let 
\begin{equation*}
K(M,g) := \operatorname{inf} \Bigl \{ r \geq \frac{1}{100} : \operatorname{sec}(g) \leq r \Bigr \},
\end{equation*}
where $\operatorname{sec}(g)$ is the sectional curvature of $g$. Then we prove the following theorem: 
\begin{theo}
Let $(M,g)$ be a compact Riemannian manifold of dimension $n$ and let $\epsilon \in \bigl(0,\frac{4}{5\pi}\bigr)$. Then there is an $\epsilon$-net $S$ in $M$, and we have the following estimate 
\begin{equation*}
\kappa(M,g,\epsilon) \leq C(n) \cdot \operatorname{Vol}(M,g)    \cdot K(M,g)^{\frac{n}{2}} \cdot \epsilon^{-n},
\end{equation*}
where $C(n) = n \cdot 20^n \cdot \pi^{\frac{n-1}{2}}  \cdot \Gamma( \frac{n-1}{2})$.
\end{theo}
As suggested in the appendix of \cite{MR2753701}, our approach is comparison geometry. We study the situation in the model case of a constantly positively curved disk. Then we use curvature bounds and comparison theorems to translate the results to more general Riemannian manifolds. 

I thank Matthias Blank, Mikhail Katz and Roman Sauer for comments and suggestions. 

I also thank Urs Lang, who informed me about a mistake in an earlier version of the text.


\section{Some geometric preparations}
In this section we examine in some detail the special case of a positively curved ball. Let $n \geq 2$ be an integer and let $D^n_{7 \pi} \subset \rr^n$ be the ball of radius $7 \pi$. Let 
\begin{equation*}
h = dr^2 + 100 \cdot \operatorname{sin}^2 \bigl( \frac{r}{10} \bigr) ds_{n-1}^2
\end{equation*}
be the constant curvature metric with $\operatorname{sec} (h) = 1/100$ (in polar coordinates). Let $\epsilon \in (0, \frac{4}{5\pi})$ and let $\delta \leq  \frac{5\pi}{4} \epsilon$. Let $d^h$ be the metric induced on $D^n_{7 \pi}$ by $h$. Let $S \subset D^n_{7 \pi}$ be a finite $\delta$-net in $D^n_{7 \pi}$, i.e., a finite subset so that for each $x \in D^n_{7 \pi}$ there is an $s \in S$ so that $d^h(x,s) < \delta$. We want to show that the $\delta$-net $S$ in $D^n_{7 \pi}$ has the following property: Let $x$ be the center $D^n_{7 \pi}$ and let $y \in D^n_{7 \pi}$ with $d^h(x,y) < 1$. Then there is an $s \in S$ so that 
\begin{equation*}
(1 - \epsilon) \cdot d^h(x,y) <  d^h(y,s) -d^h(x,s). 
\end{equation*}
To establish this, we have to find an $s \in S$ so that the angle $\alpha$ at $x$ in the geodesic triangle $\Delta := (x,y,s)$ is large, i.e., close to $\pi$. We do this in a series of lemmas.
\begin{lemma}
Let $t \in (0,1)$. Consider the function
\begin{equation*}
f_t: (0, \frac{\pi}{2}] \rightarrow \rr \, , \, x \mapsto \frac{\operatorname{sin}(tx)}{\operatorname{sin}(x)}.
\end{equation*}
Then $f_t$ has a continuous extension to $[0, \frac{\pi}{2}]$ and is non-decreasing. 
\end{lemma}
\begin{proof}
The de l'Hopital-rules show that $f_t$ can be extended continuously to $[0, \frac{\pi}{2}]$. One can compute that the derivative $f'_t$ is non-negative on $(0,\frac{\pi}{2}]$. This finishes the proof.
\end{proof}
For the rest of the section, let $x$ denote the center of $D^n_{7 \pi}$  and let $y \in D^n_{7 \pi}$ be a point with $d^h ( x,y) < 1$. 
\begin{lemma}
Let $\gamma_{x,y}$ be the maximal geodesic through $x$ and $y$. Assume that $\gamma_{x,y}$ is parametrized by arc-length and that $\gamma_{x,y} (0) = x$ and $\gamma_{x,y}(r)= y$ for some $r < 0$. Assume that $\gamma_{x,y}$ does not meet $S$. Then there is a $z \in D^n_{7 \pi}$ with $z = \gamma_{x,y}(t)$ for a $t \in (5 \pi - 2 \delta , 5 \pi + 2 \delta)$ so that the following holds. There is an $s \in S$ with $d(z,s) < \delta$, and so that the geodesic through $s$ and $z$ meets $\gamma_{x,y}$ in $z$ in a right angle.
\end{lemma}
\begin{proof}
Let $z' := \gamma_{x,y} ( 5 \pi)$. Since $S$ is $\delta$-net, there is $s \in S$ with $d^h(z' , s) < \delta$. There is a geodesic $\gamma$ emanating from $s$ and meeting $\gamma_{x,y}$ in a right angle. Let $z$ be the intersection point of $\gamma$ and $\gamma_{x,y}$. Then we have $d^h(s,z) < \delta$. 
The triangle inequality shows that $d^h(z,z') < 2 \delta$, i.e., there is a $t \in ( 5 \pi - 2 \delta , 5 \pi + 2 \delta)$ so that $\gamma_{x,y}(t) = z$.  
\end{proof}
In the next lemma, we prove that we can choose $s \in S$ so that the angle at $x$ in $\Delta$ is large.
\begin{lemma}
Let $A(\epsilon) := \operatorname{sin} \bigl( (1-\epsilon) \frac{\pi}{2} \bigr)$. Then there is an $s \in S$ so that when $\alpha$ denotes the angle at $x$ in the triangle $\Delta = (x,y,s)$, we have 
\begin{equation*}
\operatorname{cos} ( \alpha) \leq - A(\epsilon).
\end{equation*}
\end{lemma}
\begin{proof}
Choose $s \in S$ and $z \in D^n_{7 \pi}$ as in Lemma 3.2. Consider the triangle $\Delta' := (x,z,s)$. Let $\beta$ be the angle at $x$ in $\Delta'$. Let $d := d^h(z,s)$ and $b := d^h(x,s)$.  By the spherical law of the sines we have $\operatorname{sin}(\beta) =  \operatorname{sin}(\frac{d}{10}) / \operatorname{sin}(\frac {b}{10}) $.
Note that by construction the angle at $z$ in $\Delta'$ is a right angle and therefore does not enter into the equation. This gives $$\beta = \operatorname{arcsin} \Bigl( \frac{\operatorname{sin}( \frac{d}{10} ) }{\operatorname{sin} ( \frac{d}{10} )} \Bigr).$$
Observe that, since $x$ is a point on the geodesic from $y$ to $z$, we have $\alpha = \pi - \beta$. Using the estimates $\operatorname{sin} \bigl( \frac{d}{10} \bigr) \leq \operatorname{sin} \bigl( \frac{\delta} {10} \bigr) $ and $\operatorname{sin} \bigl( \frac{b}{10} \bigr) \geq \operatorname{sin} \bigl( \frac{\pi}{2} - \frac{\delta}{5}\bigr) = \operatorname{cos} \bigl( \frac{\delta} {5} \bigr)$,
and our assumption $\delta \leq \frac{5 \pi}{4} \cdot \epsilon$, we can compute:
\begin{align*}
\operatorname{cos}(\alpha) &=  \operatorname{cos} ( \pi - \beta) = - \operatorname{cos}(\beta ) 
  =  - \sqrt{ 1 - \frac{ \operatorname{sin}^2 ( \frac{d}{10} ) }{ \operatorname{sin}^2(\frac{b}{10})} } \\
 & \leq - \sqrt { 1 - \frac{\operatorname{sin}^2( \frac{\delta}{10}) }{\operatorname{cos}^2( \frac{\delta}{5}) }}   \leq - \sqrt{ \operatorname{cos}^2 \bigl( \frac{\delta}{5} \bigr) - \operatorname{sin}^2 \bigl( \frac{\delta}{5} \bigr) } \\ & = - \sqrt{ \operatorname{cos} \bigl( \frac{2 \delta}{5} \bigr) } \leq - \operatorname{cos} \bigl( \frac{2 \delta }{5} \bigr) \leq  - \operatorname{cos}( \frac{\pi}{2} \epsilon ) \\ &= - \operatorname{sin} \bigl( \frac{\pi}{2} ( 1 - \epsilon) \bigr) = -A(\epsilon) .  \qedhere 
\end{align*}
\end{proof}
We can now prove the main result of the section.
\begin{lemma}
There is an $s \in S$ so that 
\begin{equation*}
(1- \epsilon) d^h(x,y) \leq  d^h(y,s) - d^h(x,s).
\end{equation*}
\end{lemma}
\begin{proof}
Let $a:= d^h(x,y)$. According to Lemma 2.1 and Lemma 2.2, we can choose an $s \in S$ so that the angle $\alpha$ at $x$ in the triangle $\Delta = (x,y,s)$ satisfies $\operatorname{cos}(\alpha) \leq - A(\epsilon)$. Lemma 2.1 implies $\operatorname{cos} ( \alpha) \leq - \operatorname{sin} \bigl( (1- \epsilon) \frac{a}{10} \bigr) / \operatorname{sin} ( \frac{a}{10})$.   
Set $b := d^h(x,s)$ and $c := d^h(y,s)$. The spherical law of cosines now gives
\begin{align*}
\operatorname{cos}(\frac{c}{10}) =& \operatorname{cos}(\frac{a}{10}) \operatorname{cos}(\frac{b}{10}) + \operatorname{sin}( \frac{a}{10}) \operatorname{sin}(\frac{b}{10}) \cdot \operatorname{cos}(\alpha) \\  \leq& \operatorname{cos}\bigl(  (1 - \epsilon) \frac{a}{10} \bigr) \operatorname{cos}(\frac{b}{10}) - \operatorname{sin} \bigl( (1 - \epsilon)\frac{a}{10}\bigr) \operatorname{sin}(\frac{b}{10}) \\ =& \operatorname{cos} \bigl( \frac{(1- \epsilon) a + b}{10} \bigr).
\end{align*}
We conclude that $c \geq (1-\epsilon) \cdot  a + b$. It follows that $(1- \epsilon) \cdot  a \leq c -b$,
which proves the lemma.
\end{proof}
\section{An upper bound for $\kappa(M,g)$}
We now consider the general situation of a closed Riemannian manifold $(M,g)$. To find an upper bound for $\kappa(M,g,\epsilon)$, we shall first look at the special case where we assume a suitable bound for the sectional curvature $\operatorname{sec}(g)$.
 \begin{lemma}
Let $(M,g)$ be a closed Riemannian manifold with $\operatorname{sec}(g) \leq \frac{1}{100}.$ Let $\epsilon \in (0, \frac{4} {5 \pi})$ and set $\delta := \frac{1}{2} \epsilon \leq  \frac{5}{4} \pi \cdot \epsilon$. Let $S$ be a $\delta$-net in $(M,g)$. Then $S$ is $\epsilon$-good.
 \end{lemma}
 \begin{proof}
 Let $x,y \in M$. Assume first that $d^g (x,y) \geq 1$. Then choose an $s \in S$ so that $d^g(y,s) < \delta$. We compute
 \begin{align*}
 (1- \epsilon) \cdot d^g(x,y) \leq & d^g(x,y) - \epsilon \leq d^g(x,y) - 2 d^g(y,s) \\ \leq & d^g(x,s) + d^g(y,s) - 2 d^g(y,s) \leq d^g(x,s) - d^g(y,s). 
 \end{align*}
Now assume that $d^g(x,y) < 1$.  
Since $\operatorname{sec}(g) \leq \frac{1}{100}$, the estimate for the injectivity radius in \cite{MR2243772} (page 178) shows that $\operatorname{inj} (g) \geq 10 \cdot \pi > 7 \pi$.  Let $D^n_{7 \pi} \subset T_xM$ be the ball of radius $7 \pi$. Then 
 \begin{equation*}
 \operatorname{exp}_x : D^n_{7 \pi} \rightarrow \operatorname{exp}(D^n_{7 \pi}) =: B_{7 \pi}(x)
 \end{equation*}
 is a diffeomorphism. \newline \indent Put the constant positive curvature metric
 $h = dr^2 + 100 \operatorname{sin}^2( \frac{r}{10} ) d s^2_{n-1}$ on $D^n_{7 \pi}$. Then 
$ \operatorname{exp}_x: ( D^n_{7 \pi} , h) \rightarrow (B_{7 \pi} , g)$
 is a radial isometry. Since $\operatorname{sec}(g) \leq \frac{1}{100} = \operatorname{sec}(h)$, Rauch's comparison theorem (see Theorem 4.1 in \cite{MR1013810}) applies to say that $\operatorname{exp}_x$ is distance non-decreasing with respect to $d^h$ and $d^g$. Indeed, let $u',w' \in D^n_{7\pi}$. Let $u := \operatorname{exp}_x(u')$ and $w := \operatorname{exp}_x(w')$. Let $x'$ be the center of $D^n_{7\pi}$ so that $\operatorname{exp}_x(x') = x$. Since $\operatorname{exp}_x$ is a radial isometry, we have that $d^h(x',u') = d^g(x,u)$ and $d^h(x',w') = d^g(x,w)$. Furthermore, the angles at $x'$ and $x$ in the geodesic triangels $(x',u',w')$ and $(x,u,w)$ agree. In this situation, Rauch's comparison theorem says that $d^g(u,w) \geq d^h(u',w')$.  We conclude that $S' := \operatorname{exp}^{-1}(S \cap B_{7 \pi})$ is a $\delta$-net in $(D^n_{7 \pi}, d^h)$. 

Now let $y' \in D^n_{7 \pi}$ be such that $\operatorname{exp}_x(y') = y$. Then $d^h(x',y') < 1$, and by Lemma 2.5 there is an $s' \in S'$ so that 
$(1 - \epsilon) \cdot d^h(x' ,y') \leq d^h(y', s') - d^h(x',s').$
Let $s := \operatorname{exp}_x(s') \in S$. 
Since $\operatorname{exp}_x$ is a radial isometry and distance non-decreasing, we have
\begin{equation*}
d^h(x',y') =  d^g(x,y) \, , \, d^h(x',s') = d^g(x,s) \mbox{ and } d^h(y',s') \leq d^g(y,s).
\end{equation*}
 It follows that 
 \begin{align*}
 (1 - \epsilon) \cdot d^g(x,y) &= (1 - \epsilon) \cdot d^h(x',y') \leq d^h(y',s') - d^h(x',s') \\ 
 & \leq d^g(y,s) - d^g(x,s).  \qedhere
 \end{align*} 
 \end{proof}
 We can now establish an upper bound for $\kappa(M,g,\epsilon)$ in the case where there are suitable curvature bounds on $g$. 
\begin{lemma}
Let $(M,g)$ be a closed Riemannian manifold of dimension $n$ and let $\epsilon \in (0, \frac{4}{5 \pi})$. Assume that
$\operatorname{sec}(g) \leq \frac{1}{100}.$ Then:
\begin{equation*}
\kappa(M,g, \epsilon) \leq \frac{ \operatorname{Vol}(M,g) }{ \operatorname{Vol}( S^{n-1}) } \cdot 2n \cdot \pi^{n-1} \cdot  \bigl( \frac{2}{\epsilon} \bigr)^n.
\end{equation*}
\end{lemma}
\begin{proof}
Let $\delta = \frac{1}{2} \cdot \epsilon$. Then by Lemma 3.1 any $\delta$-net in $(M,g)$ is $\epsilon$-good. Choose a $\delta$-net $S$ in $(M,g)$ so that the balls $\bigl \{ B_{\frac{\delta}{2} } (s) \bigr \}_{s \in S}$ are disjoint. Let $$V := \operatorname{inf} \bigl \{ \operatorname{Vol} (B_{\frac{\delta}{2}}(s),g) : s \in S \bigr \}.$$ Then 
\begin{equation*}
|S| \leq  \frac{1}{V} \operatorname{Vol} (M,g).
\end{equation*}
Since $\operatorname{sec}(g) \leq 1/100$, the volume of $\frac{\delta}{2}$-balls in a space with constant curvature $1/100$ provides a lower bound for $V$. Again, let $h$ denote the metric $dr^2 + 100 \operatorname{sin}^2 (\frac{r}{10})ds^2_{n-1}$ on $D^n_{\frac{\delta}{2}}$. Then we compute 
\begin{align*}
V \geq \operatorname{Vol}( D^n_{ \frac{\delta}{2} }, h) &= \operatorname{Vol} (S^{n-1}) \cdot \int_0^{\frac{\delta}{2}} 10^{n-1} \cdot \operatorname{sin}^{n-1} \bigl(\frac{r}{10} \bigr) dr \\
&\geq \frac{ 2^{n-1} \operatorname{Vol}( S^{n-1} )}{\pi^{n-1}} \cdot \int_0^{\frac{\delta}{2}} r^{n-1} dr \\
&= \frac{ \operatorname{Vol}(S^{n-1}) }{ 2 n \cdot \pi^{n-1}} \cdot \delta^{n} = \frac{ \operatorname{Vol}(S^{n-1})}{2n \cdot \pi^{n-1}} \cdot \bigl(\frac{\epsilon}{2}\bigr)^n.
\end{align*}
It follows that
\begin{equation*}
\kappa(M,g, \epsilon) \leq \frac{ \operatorname{Vol}(M,g)}{\operatorname{Vol}(S^{n-1})} \cdot 2n \cdot \pi^{n-1} \cdot \bigl(\frac{2}{\epsilon} \bigr)^n. \qedhere
\end{equation*}
\end{proof}
Now we can give a proof of our main result, establishing an upper bound for $\kappa(M,g,\epsilon)$ in case that $(M,g)$ is a compact Riemannian manifold.

\begin{proof}[Proof of Theorem 1.1]
The idea is to scale the metric $g$ so that the scaled metric satisfies suitable curvature bounds. Let 
\begin{equation*}
t = t(M,g) := 10 \cdot \sqrt{ K(M,g) }.
\end{equation*}
Then set $\overline{g} := t^2 \cdot g$. It follows that
\begin{equation*}
\operatorname{sec}(\overline{g} ) = t^{-2} \cdot \operatorname{sec}(g)  \leq  \frac{1}{100},
\end{equation*}
i.e., Lemma 3.2 is applicable to $\overline{g}$. 
Now, since the inequality 
\begin{equation*}
(1-\epsilon) \cdot d^g(x,y) \leq | d^g(x,s) - d^g(y,s) | 
\end{equation*}
for $x,y,s \in M$ is invariant under scaling of $g$, it follows from Lemma 3.1 that there is an $\epsilon$-good net $S$ in $(M,g)$, and we can use Lemma 3.2 to deduce
\begin{align*}
\kappa(M,g,\epsilon) &= \kappa(M,\overline{g} ,\epsilon) \leq \frac{ \operatorname{Vol}(M, \overline{g}) }{ \operatorname{Vol}(S^{n-1}) } \cdot 2n \cdot \pi^{n-1} \cdot \bigl(\frac{2}{\epsilon})^n \\ &= \frac{\operatorname{Vol}(M,g) }{ \operatorname{Vol}(S^{n-1}) } \cdot 2n \cdot \pi^{n-1} \cdot 20^n \cdot K(M,g)^{\frac{n}{2}} \cdot \bigl( \frac{1}{\epsilon})^n. 
\end{align*}
Now use $\operatorname{Vol}(S^{n-1}) = 2\pi^{\frac{n-1}{2}} / \Gamma( \frac{n-1}{2})$ to finish the proof.
\end{proof}

\bibliography{cobordismthm}{}
\bibliographystyle{plain}

\end{document}